\documentclass{amsart}
\usepackage[utf8]{inputenc}
\usepackage[english]{babel}
\usepackage[T1]{fontenc}
\usepackage[pdftex]{hyperref}
\usepackage[square,numbers]{natbib}
\usepackage{soul}
\usepackage{hyperref}
\usepackage{url}
\usepackage{amsmath}
\usepackage{amsthm}
\usepackage{stmaryrd}
\usepackage{amsfonts}
\usepackage{amssymb}
\usepackage{mathrsfs}
\usepackage{setspace}
\usepackage{enumitem}
\usepackage{pdfpages}
\usepackage[all]{xy}
\newcommand{\Sum}{\displaystyle \sum}

\newcommand{\C}{\mathbb{C}_{\infty}}
\newcommand{\K}{K_{\infty}}
\newcommand{\N}{\mathbb{N}}
\newcommand{\Fq}{\mathbb{F}_{q}}
\newcommand{\F}{\mathbb{F}}
\newcommand{\Q}{\mathbb{Q}}
\newcommand{\Z}{\mathbb{Z}}
\newcommand{\R}{\mathbb{R}}
\newcommand{\Ii}{\mathcal{I}_\infty}
\newcommand{\Pj}{\mathbb{P}}

\newcommand{\T}{\mathcal{T}}

\newcommand{\Ht}{\underline{H}}
\newcommand{\Hzo}{\underline{H}_{!}(\T,\Z)^\Gamma}

\renewcommand{\O}{O_{\infty}}
\newcommand{\Gl}{\mathrm{GL}_2}

\newcommand{\Prod}{\displaystyle \prod}

\DeclareMathOperator{\card}{\mathrm{card}}
\DeclareMathOperator{\Tor}{\mathrm{Tor}}

\DeclareMathOperator{\pgcd}{\mathrm{gcd}}

\DeclareMathOperator{\Hom}{\mathrm{Hom}}

\theoremstyle{plain}

\newtheorem{def }{Definition}

\newtheorem{prop}[def ]{Proposition}
\newtheorem{theo}[def ]{Theorem}
\newtheorem{lem }[def ]{Lemma}
\newtheorem{coro}[def ]{Corollary}
\theoremstyle{definition}

\newtheorem*{rem}{Remark}

\title[Parametrization of elliptic curves by Drinfeld modular curves]{Explicit computation of the modular parametrization of elliptic curves over function fields by Drinfeld modular curves}

\date{}

\author[V. Petit]{Valentin Petit}
\address{Laboratoire de math\'ematiques de Besan\c con, Universit\'e Bourgogne Franche-Comt\'e, CNRS UMR 6623, 16, route de gray, 25000 Besan\c con, France}
\email{valentin.petit@univ-fcomte.fr}

\begin{document}

\begin{abstract}
Let $q$ be a prime power and $E$ a non-isotrivial elliptic curve over $\Fq(T)$ given by a Weierstrass  model.
%$$
%E \colon y^2+a_1xy+a_3y = x^3+a_2x^2 + a_4x +a_6, %\quad a_i \in \F_q[T].
%$$
We survey the construction, with an explicit point of view, of the modular parametrization of $E$ by the associated Drinfeld modular curve. 
We then prove a formula that allows us to evaluate this modular parametrization at cusps and we produce an explicit method to compute these values.
Finally we illustrate our results with several examples in characteristic 2 and 3.  
\end{abstract}

\subjclass[2020]{11G18, Secondary 11G09, 11F12, 14H52.} 
\keywords{Elliptic curves, Function Fields, Modular parametrization, Harmonic cochains, Theta functions, Drinfeld modular curves.}

\maketitle

\section{Introduction and motivation} 
Let $q$ be a prime power and $E$ be an elliptic curve defined over $\F_q(T)$ given by a Weierstrass model
$$
E \colon y^2+a_1xy+a_3y = x^3+a_2x^2 + a_4x +a_6, \quad a_i \in \F_q[T].
$$
We assume that $E$ is non-isotrivial and has split multiplicative reduction at a place, say $\infty:=1/T$. By the work of Drinfeld, Grothendieck, Jacquet and Langlands (\citep{drinfel1974elliptic,jacquetLang}), there exists a "modular
parametrization" $\Phi \colon \overline{M}_{\Gamma_0(n)} \rightarrow E$, where $\overline{M}_{\Gamma_0(n)}$ is the 
Drinfeld modular curve over $\Fq(T)$ associated to the Hecke congruence subgroup
$$
\Gamma_0(n) = \left\{ \begin{pmatrix} a & b \\ c & d \end{pmatrix} \in \Gl(\F_q[T]) \colon c\equiv 0 \bmod n \right\},
$$
and $n$ is the finite part of the conductor of $E$. The map $\Phi$ has been studied and described by several authors, especially Gekeler-Reversat \citep{gekelerrevjac96} and Gekeler \citep{gekeler1995analytical,MR1630606}. The main goal 
of this article is to explain how this modular parametrization can be explicitly computed and to deduce  some 
arithmetical properties from this. In particular, we describe a method to compute \emph{exactly} the images $ \Phi(c) \in E$ of 
the cusps $c \in \overline{M}_{\Gamma_0(n)}$.

The situation above has to be compared with the classical one in characteristic zero. In that case, if $F$ is an elliptic 
curve defined over $\Q$ with conductor $N$, by a series of works culminating with \citep{taylor1995ring,breuil2001modularity}, 
there exists a modular parametrization given by the composition
$\varphi \colon X_0(N) \overset{\phi}{\rightarrow} \mathbb{C}/\Lambda \overset{\wp}{\rightarrow} F(\mathbb{C})$, where $\Lambda$ is 
the period lattice of $F$, the map $\phi$ is roughly 
speaking given by the differential form of the weight-2 modular form $f(\tau) = \sum_{n\geq1} a_n e^{2in\pi \tau}$ associated 
to $F$, and the isomorphism $\wp$ is given by the Weierstrass function and its derivative. In the literature there are many investigations on the explicit description of the map $\phi$ and the computation of its arithmetical 
invariants: its degree (\citep{zagier1985modular, cremona1995computing,delaunay2003computing, watkins2002computing}), its critical points (%[Mazur],
\citep{chen2016computing,delaunay2005critical}),
its evaluation at cusps (\citep{brunault2016ramification,wuthrichmodsymbol}), its explicit evaluation and applications to Heegner points  (\citep{MR3098131}).

In the case of Drinfeld modular curves, the fact that $E$ has split multiplicative reduction at $\infty$ implies 
 the existence of a Tate parameter $t \in \C^*$ and an isomorphism $E(\C) \simeq \C^*/t^\Z$ where $\C$ 
 is a completion of an algebraic closure of $\Fq((1/T))$. Hence on the one hand, the modular parametrization is given as in the classical case by a map 
 $$
\Phi \colon \overline{M}_{\Gamma_0(n)} \rightarrow \C^*/t^\Z
 $$
 where $\Phi$ will be defined in Section \ref{Modpar}. On the other hand, unlike the characteristic zero situation, the degree of $\Phi$ is easily computable 
 using a formula of Gekeler \citep{gekeler1995analytical} recalled in Theorem~\ref{deg modular}. The definition of $\Phi$ also requires a lot of theoretical material that we will recall for the self-containedness of this text.

The organization of the paper is as follows. In the notations and in Sections~\ref{BTt} to~\ref{Thetafct}, we give a survey of classical theoretical results on Drinfeld modular curves and related objects which are used for the construction of $\Phi$: the Bruhat-Tits tree, harmonic cochains, theta functions. Section~\ref{proofTH} is devoted to Theorem~\ref{Prodcusp}, which states that the expression of the map $\Phi$ using a theta function, which originally holds outside of the cusps, may also be evaluated at the cusps. This is our first key step for evaluating explicitly $\Phi$ at a cusp. As far as we know, there is no similar result in the characteristic zero setting (however numerical experimentations suggest that it should be the case). In Section~\ref{HeckS}, we recall the action of Hecke operators on different objects involved in our study. In Section~\ref{Modpar}, we use Hecke operators to give an explicit upper bound for the order of the image of a cusp by $\Phi$ (Theorem~\ref{ptetors}): this is reminiscent of the Manin-Drinfeld theorem, whose analogue in our context was proved by Gekeler \citep{MR1487469}, but our bound is explicit and it is the second key step in our method. Finally in Section~\ref{Appl} we illustrate our explicit description of the modular parametrization by two examples in characteristic 2 and 3 with small conductors (note that explicit examples in characteristic $>2$ seem to be rare in the literature). 

\subsection*{Acknowledgements}
The author is a member of the Laboratoire de math\'ematiques de Besan\c con (LmB) which receives support from the EIPHI Graduate School (contract ANR-17-EURE-0002).

\subsection*{Notations}
Let $p$ be a prime number, $q$ a power of $p$, and $\Fq$ a finite field with $q$ elements. We denote by $A=\Fq[T]$ the ring of polynomials with coefficients in $\Fq$ and $K=\Fq(T)$ 
the fraction field of $A$. Let $\K=\Fq((\frac{1}{T}))$ be the completion of $K$ with respect to the norm $|P|=q^{\deg(P)}$. The element $\pi=\frac{1}{T}\in K$ is a uniformizer of $\K$ and let 
$\O=\Fq[[\pi]]$ be the ring of integers of $\K$. We denote by $\C$ a completion of $\overline{\K}$ where $\overline{\K}$ is an algebraic closure of $\K$. Finally let $\Omega=\C-\K$ be the Drinfeld upper 
half-plane. The imaginary part map on $\Omega$, denoted $|\cdot|_i$, is defined by
$$
|z|_i=\inf \{|z-x|, \ x \in \K\}.
$$ 
Recall that the group $\Gl(\K)$ acts on $\Omega$ by fractional linear transformations: $\begin{pmatrix}
a&b\\c&d
\end{pmatrix}z=\dfrac{az+b}{cz+d}$. The imaginary part satisfies the relation 
$$
|\gamma z|_i=\dfrac{|z|_i}{|cz+d|^2} \qquad  \left(z \in  \Omega, \ \gamma=\begin{pmatrix}
a&b\\c&d
\end{pmatrix} \in \Gl(\K) \right).
$$
Any arithmetic subgroup $\Gamma \subset \Gl(A)$ acts on $\Omega$ with finite stabilizers: it implies that the quotient $\Gamma \backslash \Omega$ has a structure of an analytic space over $\K$. There exists a smooth irreducible affine algebraic curve $M_\Gamma$ defined over $\K$ whose underlying analytic space is canonically isomorphic to $\Gamma \backslash \Omega$. Let $\overline{M}_\Gamma$ be the smooth projective model of $M_\Gamma$. The algebraic curve $\overline{M}_\Gamma$ is called the Drinfeld modular curve associated to $\Gamma$.
\section{The Bruhat-Tits tree}\label{BTt}
For Sections \ref{BTt}, \ref{HarmC}, and \ref{Thetafct}, we mainly use the references \citep{gekelerrevjac96}, \citep{gekeler1995analytical}, and \citep{gekeler1995fundamental}.
We denote by $\mathcal{Z}$ the center of $\Gl(\K)$ and $\Ii$ the Iwahori subgroup defined by 
$$
\Ii=\left\{\begin{pmatrix}
a&b \\c&d
\end{pmatrix} \in \Gl(\O), \ c \equiv 0\bmod \pi\right\}.
$$
We recall a description of the Bruhat-Tits tree in terms of the group $\Gl(\K)$ as in \citep[II, 1.3]{serretrees}, \citep{gekeler1995fundamental}.%and \citep[Chapter 2]{armana2008torsion}.
\begin{def }
The Bruhat-Tits tree $\T$ is the combinatorial graph with set of vertices $X(\T)=\Gl(\K)/\mathcal{Z}\Gl(\O)$ and set of oriented edges 
$Y(\T)=\Gl(\K)/\mathcal{Z}\Ii$. We associate to $\T$ the following maps
$$
\begin{array}{lccc}
o \colon & Y(\T) &\longrightarrow & X(\T) \\
& M \bmod \mathcal{Z}\Ii &\longmapsto & M \bmod \mathcal{Z} \Gl(\O)
\end{array}
$$
and 
$$
\begin{array}{lccc}
\overline{\cdot} \colon & Y(\T) &\longrightarrow & Y(\T) \\
& e &\longmapsto & \overline{e} = e \begin{pmatrix}
0&1\\\pi&0
\end{pmatrix}.
\end{array}
$$
The map $o$ is the canonical map that associates to an edge its origin. The map $e \mapsto \overline{e}$ associates to an edge $e$ its opposite edge denoted by~$\overline{e}$.
\end{def }
The graph $\T$ is a tree in the sense of Serre (\citep[I.2]{serretrees}) and it is a $(q+1)$-regular graph. The group $\Gl(\K)$ acts on 
the left on $\T$ in a natural way. Any arithmetic subgroup $\Gamma \subset \Gl(A)$ acts on $\T$ without inversion. Hence we can consider the quotient graph $\Gamma \backslash \T$ 
whose set of vertices is $X(\Gamma \backslash \T)=\Gamma \backslash X(\T)$ 
and set of edges is $Y(\Gamma \backslash \T)=\Gamma \backslash Y(\T)$. By the work of Serre (\citep[II.1.2 , II.1.3]{serretrees}), the graph $\Gamma \backslash \T$ is the disjoint 
union of a finite graph denoted by $(\Gamma \backslash \T)^0$ and a 
finite number of half-lines called the ends of $\Gamma \backslash \T$.
%labelled by the cusps $\Gamma \backslash \Pj^1(K)$. %Moreover each half-line is in bijection 
%with $\bullet-\bullet-\bullet- \bullet-...$.  
%The cusps (also called ends) of $\T$ can be identified with the 
%set of cusps of $\overline{M}_{\Gamma}$.
In fact we have the following statement (see \citep{MR1630606}).
\begin{prop}
There are canonical bijections between the following sets:
\begin{enumerate}
\item the ends of $\Gamma \backslash \T$,
\item the cusps of $\overline{M}_\Gamma$,
\item the orbits $\Gamma \backslash \Pj^1(K)$ of $\Pj^1(K)$ under $\Gamma$.
\end{enumerate}
\end{prop}
Let $(v_k)_{k \in \N}$ be the vertex of $\T$ corresponding to the classes of the matrices
$\begin{pmatrix}
\pi^{-k}&0\\0&1
\end{pmatrix}_{k \in \N}$ and let $e_k$ be the edge from $v_k$ to $v_{k+1}$. The end $\infty=(v_0,v_1,v_2,...)$ of $\T$ defines an orientation 
on $\T$ i.e. a decomposition $Y(\T)=Y^+(\T)\sqcup Y^-(\T)$ with $\overline{Y^+(\T)}=Y^-(\T)$. An edge $e$ is called positive if it
points towards the end $\infty$. The following statement is well-known and can be found for instance in \citep[p.~371]{gekeler1995improper}.
\begin{lem }\label{repBTtree}
The set of matrices $\begin{pmatrix}
\pi^j&y \\0&1
\end{pmatrix}$, with $j \in \Z$ and $y \in K_\infty \bmod \pi^j \O$, is a system of representatives for $X(\T)$ and $Y^{+}(\T)$.
\end{lem }
\section{Harmonic cochains}\label{HarmC}
\begin{def }\label{def-cochain}
Let $B$ be an abelian group. A map $\varphi \colon Y(\T) \rightarrow B$ is said to be a ($B$-valued) harmonic cochain if it satisfies the two following conditions:
\begin{enumerate}
\item for all $e \in Y(\T)$, $\varphi(\overline{e})=-\varphi(e)$,
\item for all $v \in X(\T)$, $\Sum_{o(e)=v}\varphi(e)=0$.
\end{enumerate}
Furthermore if for all $\gamma \in \Gamma$ and $e \in Y(\T)$ we have $\varphi(\gamma e)=\varphi(e)$, we say that $\varphi$ is $\Gamma$-invariant. 
\end{def }
The additive group of $B$-valued harmonic cochains is denoted by $\Ht(\T,B)$ and its subgroup of $\Gamma$-invariant harmonic cochains by  $\Ht(\T,B)^\Gamma$.

Elements of $\Ht(\T,B)^\Gamma$ can be considered as maps defined on the edges of the quotient graph 
$\Gamma \backslash \T$. Let $v \in X(\T)$ be a vertex and $\widetilde{v}$ its equivalence class modulo~$\Gamma$. The stabilizer of $v$, which is denoted by $\Gamma_v$, acts on the set 
$$
\{e \in Y(\T), \ o(e)=v\}.
$$
For $e \in Y(\T)$, let $\Gamma_e$ be its stabilizer. The length of the orbit of $e$ is $m(e)=[\Gamma_v: \Gamma_e]$ and this number only depends on the image $\tilde{e}$ of $e$ in $\Gamma \backslash \T$. Viewing $\varphi \in \Ht(\T,B)^\Gamma$ as a map on 
$Y(\Gamma \backslash \T)$, the sum condition in Definition~\ref{def-cochain} translates into
$$
\Sum_{\underset{o(\tilde{e})=\tilde{v}}{\tilde{e} \in Y(\Gamma\backslash \T)}}m(\tilde{e})\varphi(\tilde{e})=0.
$$
\begin{def }
A harmonic cochain $\varphi \in \Ht(\T, B)^\Gamma$ is said to be cuspidal if it has a finite support modulo $\Gamma$. 
\end{def }
We denote by $\Ht_!(\T,B)^\Gamma$ the subgroup of cuspidal harmonic cochains.
On $\Ht_!(\T,\mathbb{C})^\Gamma$, the Petersson inner product is defined by
\begin{equation}\label{Peter}
\langle \varphi, \psi\rangle=\Sum_{\tilde{e}  \in Y(\Gamma \backslash \T)}\dfrac{q-1}{2}|\Gamma_e|^{-1}\varphi(\tilde{e})\overline{\psi(\tilde{e})} \qquad (\varphi, \psi \in \Hzo).
\end{equation}
%
%Recall that $n\in A$ is the finite part of the conductor of $E$.
Let $n\in A$. If $n' \in A$ is a polynomial dividing $n$ then for each monic 
divisor $a$ of $n/n'$, we have an embedding $i_{a,n'}\colon \Ht_!(\T,B)^{\Gamma_0(n')} \rightarrow \Ht_!(\T,B)^{\Gamma_0(n)}$ 
given by
$$
i_{a,n'}(\varphi)(e)=\varphi \left( \begin{pmatrix}
a&0\\0&1
\end{pmatrix}e \right) \qquad (\varphi \in \Ht_!(\T,B)^{\Gamma_0(n')}, \ e \in Y(\T)).
$$
\begin{def }\label{newfr}
Let $\Ht_!^{\mathrm{new}}(\T,\Q)^{\Gamma_0(n)} $ be the orthogonal complement in $\Ht_!(\T,\Q)$, with respect to the Petersson 
inner product, of the images of all the $i_{a,n'}$ where $n'$ runs through the proper divisors of $n$. Let 
$\Ht_!^{ \mathrm{new}}(\T,\Z)^{\Gamma_0(n)}=\Ht_!(\T,\Z)^{\Gamma_0(n)} \cap \Ht_!^{\mathrm{new}}(\T,\Q)^{\Gamma_0(n)} $. 
A harmonic cochain $\varphi \in \Ht_!(\T,\Z)^{\Gamma_0(n)} $ is called a newform if 
$\varphi \in \Ht_!^{\mathrm{new}}(\T,\Z)^{\Gamma_0(n)}$.
\end{def }
We now introduce a map $j$ which will be used to construct the modular parametrization. Let $v,w \in X(\T)$ be two vertices. We denote by $g(v,w)$ the unique geodesic going from $v$ to $w$. For $v \in X(\T)$,  
$e \in Y(\T)$ and $\gamma, \alpha \in \Gamma$, we let
$$
\iota (e, \alpha,\gamma,v)=\begin{cases}
1 &\mathrm{if} \ \gamma e \in g(v,\alpha v), \\
-1 &\mathrm{if} \ \gamma e \in g(\alpha v,v), \\
0 & \mathrm{otherwise}.
\end{cases}
$$
Since the map $\gamma \mapsto \iota(e,\alpha, \gamma,v)$ has finite support, the quantity
$$
\varphi_{\alpha,v}(e)=\dfrac{1}{|\Gamma \cap \mathcal{Z}|}\Sum_{\gamma \in \Gamma} \iota(e,\alpha,\gamma,v)
$$
is well-defined and it is also $\Z$-valued. We let 
$\overline{\Gamma}=\Gamma^{\mathrm{ab}}/\Tor(\Gamma^{\mathrm{ab}})$, where $\Gamma^{\mathrm{ab}}$ is the abelianization 
of $\Gamma$ and $\Tor(\Gamma^{\mathrm{ab}})$ is the torsion subgroup of $\Gamma^{\mathrm{ab}}$.
%The following result 
%is stated in \citep[Lemma~3.3.3]{gekelerrevjac96}.
%
\begin{lem }[{\citep[Lemma~3.3.3]{gekelerrevjac96}}]
The functions $\varphi_{\alpha,v}: Y(\T) \rightarrow \Z$ have the following properties:
\begin{enumerate}
\item For all $v \in X(\T)$ and all $\alpha \in \Gamma$, $\varphi_{\alpha,v} \in \Hzo$.
\item The function $\varphi_\alpha=\varphi_{\alpha,v}$ is independent of the choice of $v \in X(\T)$. 
\item The map $\alpha \mapsto \varphi_\alpha$ induces a group homomorphism $j$ from $\overline{\Gamma} $ to $\Hzo$.
\item The map $j\colon \overline{\Gamma} \rightarrow \Hzo$ is one-to-one and with finite cokernel.
\end{enumerate}
\end{lem }
Moreover as proved by Gekeler-Nonnengardt in \citep[Theorem 3.3]{gekeler1995fundamental}, we have in the case of the subgroup $\Gamma_0(n)$:
\begin{theo}\label{isomj}
Let $n \in A$ be a non constant polynomial. If $\Gamma=\Gamma_0(n)$ then the map $j\colon \overline{\Gamma} \rightarrow \Hzo$ is an isomorphism.
\end{theo}
Let $g(\Gamma)$ be the genus of the modular curve $\overline{M}_\Gamma$. For an abelian group $B$, let  denote by 
$\mathrm{rk}_\Z(B)$ its rank. We have the following equalities (see \citep{gekeler1995analytical}):
$$
\mathrm{rk}_{\Z}(\overline{\Gamma})=\mathrm{rk}_{\Z}(\Hzo)=g(\Gamma).
$$
\section{Theta functions for \texorpdfstring{$\Gamma$}{Gamma}}\label{Thetafct}
\begin{def }[{\citep[Section 5]{gekelerrevjac96}}]
A holomorphic theta function (resp. meromorphic theta function) for $\Gamma$ is a holomorphic  function on $\Omega$ without zeros and poles on $\Omega$ 
and at the cusps (resp. without zeros and poles at cusps), and satisfying 
$$
f(\gamma z)=c_f(\gamma)f(z), \quad \mbox{ for all } z \in \Omega \mbox{ and } \gamma \in \Gamma,
$$
with $c_f(\gamma) \in \C^{*}$ independent of $z$. The map $c_f \colon \Gamma \rightarrow \C^{*}$ is called the multiplier of $f$.
\end{def }
Let $m$ a positive integer. We set 
$U_m=\{z \in \Omega, |z| \leq q^m,\ |z|_i \geq q^{-m} \}$. The set $\Omega$ is a rigid analytic subspace in $\Pj^1(\C)$ and $\Omega=\underset{m \geq 1}{\bigcup}U_m$ is an admissible covering.

For the rest of this section, we fix two elements $\omega, \eta \in \Omega$.
\begin{def }
 We put $\widetilde{\Gamma}=\Gamma/(\Gamma \cap \mathcal{Z})$ and
\begin{equation}\label{prod-theta}
\theta(\omega,\eta,z)=\Prod_{\gamma \in \widetilde{\Gamma}} \dfrac{z-\gamma \omega}{z-\gamma \eta}.
\end{equation}
\end{def }

Convergence and other properties of these theta functions are summarized in the following statement.
\begin{theo}[{\citep[Proposition 5.2.3, Theorem 5.4.1 and Proposition 5.4.12]{gekelerrevjac96}}]\label{prétheta}\
\begin{enumerate}%[label=\roman*]
\item The product \eqref{prod-theta} converges locally uniformly on $\Omega$ and defines a meromorphic theta function for 
$\Gamma$. Moreover, it has only zeros and poles at $\Gamma \omega, \, \Gamma \eta$ if $\Gamma \omega \neq \Gamma \eta$.
\item The function $\theta(\omega,\eta,.)$ satisfies $\theta(\omega,\eta,\gamma z)=c(\omega,\eta,\gamma)\theta(\omega,\eta,z)$ for any $\gamma \in \Gamma$, $z\in \Omega$, with $c(\omega,\eta,\gamma)\in \C^*$. Moreover, $c(\omega,\eta,.)$ factors over $\overline{\Gamma}$.
\item Given $\alpha \in \Gamma$, the holomorphic function \begin{equation}\label{ualphaprod}u_\alpha(z)=\theta(\omega, \alpha \omega,z) = \Prod_{\gamma \in \widetilde{\Gamma}} \dfrac{z-\gamma \omega}{z-\gamma \alpha \omega}\end{equation} is well-defined 
and independent of $\omega \in \Omega$. It only depends on the class of $\alpha \in \overline{\Gamma}$. 
\item The multiplier $c(\omega, \eta,.)$ satisfies for all $\alpha \in \Gamma$, 
$c(\omega, \eta,\alpha)=\dfrac{u_\alpha(\eta)}{u_\alpha(\omega)}$ and in particular, is holomorphic in $\omega$ and $\eta$.
\item Let $c_\alpha(.)=c(\omega,\alpha\omega,.)$ be the multiplier of $u_\alpha$. The map $(\alpha, \beta) \mapsto c_\alpha(\beta)$ 
is a symmetric bilinear map on $\overline{\Gamma} \times\overline{\Gamma}$ which takes values in $\K^{*}$. 
\end{enumerate}
\end{theo}
%Let $\omega, \eta$ be two elements of $\Omega$.
We recall that  the multiplier $c(\omega,\eta,.)$ satisfies 
(see \citep[5.4.9]{gekelerrevjac96})
\begin{equation}\label{coeffmul}
c(\omega,\eta, \alpha)=
\left\{
\begin{array}{ll}
1 &  \mathrm{if} \ \alpha \infty=\infty,  \\
\Prod_{\gamma \in \widetilde{\Gamma}} \dfrac{\alpha \infty-\gamma \omega}{\alpha \infty- \gamma \eta} &  \mathrm{otherwise.}
\end{array}
\right.
\end{equation}
Since for all $\alpha \in \Gamma$, the map $\gamma \in \Gamma \mapsto c_\alpha(\gamma) \in \C^{*}$ factors over 
$\overline{\Gamma}$, the map
$$
\begin{array}{lccc}
c \colon &\Gamma &\longrightarrow &\Hom(\overline{\Gamma}, \C^*) \\
&\alpha & \longmapsto & c_\alpha(.)
\end{array}
$$
is well-defined. Furthermore
the map $c$ factors over $\Gamma$ and gives a map $\overline{c}\colon \overline{\Gamma} \rightarrow \Hom(\overline{\Gamma}, \C^*)$.
\begin{theo}[{\citep[Proposition~2.6]{MR1487469}}]\label{thcusp}
Let $\omega, \eta$ be elements of $\Omega$. The function $\theta(\omega,\eta,.)$ has a meromorphic continuation to the boundary 
$\Pj^1(K)$ of $\Omega$. Moreover $\theta(\omega, \eta,.)$ is holomorphic without any zeros at cusps.
\end{theo}
\begin{rem}
%Let $\omega, \eta$ be two elements of $\Omega$.
The value of the meromorphic continuation of $\theta(\omega,\eta,.)$ at the cusp 
$\infty$ is 1 (this can be derived from \citep[Lemma 5.3.10]{gekelerrevjac96}).
\end{rem}
\section{Formula for the value of theta functions at a cusp}\label{proofTH}
A consequence of Theorem \ref{thcusp} is that the holomorphic function $u_\alpha$ has a holomorphic continuation to $\overline{\Omega} = \Omega\, \sqcup\,  \Pj^1(K)$ hence it is a holomorphic theta function. Furthermore, as proved in \citep[Lemma 5.3.9]{gekelerrevjac96}, for all $s \in K$ and all 
$\omega,\eta \in \Omega$, the product $\Prod_{\gamma \in \widetilde{\Gamma}}\dfrac{s-\gamma \omega}{s-\gamma \eta}$ converges.
Note that it does not necessarily imply that the image of $u_\alpha$ at a cusp $s \in K$ is given by the product formula~\eqref{ualphaprod}. 
However we  prove now that it is indeed the case.
\begin{theo}\label{Prodcusp}
Let $\alpha$ be an element of $\overline{\Gamma}$ and $s \in \Pj^1(K)-\{\infty\}$. The value of $u_\alpha$ at $s$ is given by
\begin{equation*}
u_\alpha(s)= \Prod_{\gamma \in \widetilde{\Gamma}} \dfrac{s-\gamma \omega}{s-\gamma \alpha \omega}.
\end{equation*}
\end{theo}
The rest of this section is devoted to the proof of Theorem~\ref{Prodcusp}.
Let $s \in \Pj^1(K)-\{\infty\}$. Since the product  $u_\alpha(z)$ is independent of the choice of $\omega \in \Omega$, we 
can choose $\omega \in \Omega$ such that $|\omega| \in q^{\Q-\Z}$. This implies that $|\omega|=|\omega|_i$.
We have
\begin{equation}\label{pr}
\left|\Prod_{\gamma \in \widetilde{\Gamma}}\dfrac{s-\gamma \omega}{s-\gamma \alpha \omega}-\Prod_{\gamma \in  \widetilde{\Gamma}}\dfrac{z-\gamma \omega}{z-\gamma \alpha \omega} \right|=\left| u_\alpha(z)\right|
\left| \Prod_{\gamma \in \widetilde{\Gamma}}\dfrac{(s-\gamma \omega)(z-\gamma \alpha \omega)}{(s-\gamma \alpha \omega)(z-\gamma \omega)}-1 \right|.
\end{equation} 
We let $F_{\gamma,s}(z)=\dfrac{(s-\gamma \omega)(z-\gamma \alpha \omega)}{(s-\gamma \alpha \omega)(z-\gamma \omega)}$.
\begin{lem }\label{lemcv1}
For $\gamma =\begin{pmatrix}
a&b\\c&d
\end{pmatrix} \in \widetilde{\Gamma}$, we have 
$$
F_{\gamma,s}(z)=1+\dfrac{(z-s)(\alpha \omega-\omega)\det(\gamma)}{(s-\gamma \alpha \omega)(z-\gamma \omega)(c\omega+d)(c \alpha \omega+d)}.
$$
\end{lem }
\begin{proof}
Let $\gamma=\begin{pmatrix} a&b\\c&d \end{pmatrix} \in \widetilde{\Gamma}$ and $z \in \Omega$. We first note that
\begin{equation}\label{gammaomg}
\gamma \alpha \omega-\gamma \omega= \dfrac{(\alpha \omega-\omega)\det (\gamma)}{(c\omega+d)(c \alpha \omega+d)}.
\end{equation}
Furthermore we have
\begin{align*}
\dfrac{(s-\gamma \omega)(z-\gamma \alpha \omega)}{(s-\gamma \alpha \omega)(z-\gamma \omega)} &= \dfrac{sz-s(\gamma \alpha \omega)-z(\gamma \omega)+(\gamma \omega)(\gamma \alpha \omega)}{(s-\gamma\alpha \omega)(z-\gamma \omega)} \\
&=1+\dfrac{(z-s)(\gamma \alpha \omega-\gamma \omega)}{(s-\gamma \alpha \omega)(z-\gamma \omega)} \\
&=1+\dfrac{(z-s)(\alpha \omega-\omega)\det(\gamma)}{(s-\gamma \alpha \omega)(z-\gamma \omega)(c \omega+d)(c\alpha \omega+d)}
\end{align*}
and the lemma follows.
\end{proof}
We know that $|s-\gamma \alpha \omega|\geq \dfrac{\kappa_{s,\alpha \omega}}{|c\alpha \omega+d|}$ 
(see \citep[proof of Lemma 5.3.9]{gekelerrevjac96}) where $\kappa_{s, \alpha \omega}$ is a positive real number only depending 
on $s$, $\alpha$ and $\omega$. It yields 
\begin{equation}\label{majorationFg}
|F_{\gamma,s}(z)-1|\leq \dfrac{|z-s||\alpha \omega-\omega| \kappa_{s, \alpha \omega}^{-1}}{|z-\gamma  \omega||c \omega+d|}\leq \dfrac{|z-s| C_{\omega,\alpha,s}}{|z-\gamma  \omega||c \omega+d|}
\end{equation}
where $C_{\omega, \alpha,s}$ is independent of $\gamma \in \Gamma$ and $z \in \Omega$.\medskip
\\
The goal is to prove that the limit of \eqref{pr} is $0$ as $z$ tends to $s$. From the meromorphic continuation of  $u_\alpha$ at $s$,
we deduce that $u_\alpha(z) \underset{z \rightarrow s}{\rightarrow} u_\alpha(s)$. It is then sufficient to find a sequence 
$(z_n)_{n \in \N} \subset \Omega$ such that $\underset{n \rightarrow \infty}{\lim}z_n=s$ and such that
\begin{equation}\label{prod s}
\underset{n \rightarrow \infty}{\lim} u_\alpha(z_n)=\Prod_{\gamma \in \widetilde{\Gamma}}\dfrac{s-\gamma \omega}{s-\gamma \alpha \omega}.
\end{equation}
We have 
\begin{equation}\label{ConstC}
|u_\alpha (z_n)| \leq \max(|u_\alpha(s)|,|u_\alpha(s)-u_\alpha(z_n)|) < C
\end{equation}
where $C>0$ only depends on $s$ and $\alpha$ since   $u_\alpha(z_n) \rightarrow u_\alpha(s)$. Now to prove \eqref{prod s}, 
we need to prove that for all $\varepsilon>0$, there exists $n_\varepsilon>0$ such that for all $n\geq  n_\varepsilon$, 
$|F_{\gamma ,s }(z_n)-1|< \varepsilon$.
 We choose $\zeta \in \Omega-\Gamma  \omega$ such that $|\zeta|=|\zeta|_i=1$ and we let $z_n=s+T^{-n}\zeta$. Remark that
 $|z_n|_i=q^{-n}$ for all non-negative integers $n$.
\begin{lem }\label{lemcv2}
Let $\varepsilon>0$ be a real number. There exists a non-negative integer $n_0$ such that for all $n \geq n_0$, 
if $\gamma$ satisfies $|z_n-s|_i> |\gamma \omega|_i$ then $|F_{\gamma,s}(z_n)-1|\leq \varepsilon$.
\end{lem }
\begin{proof}
Let $\varepsilon>0$ be fixed and  $n \geq 0$. First note that $|c \omega+d|\geq |c \omega|_i=|c||\omega|_i$.
Hence, if $|c|> |\omega|_i^{-1} q^n$ or if $|d|>\max\left(|c \omega|,\left(| \omega |_i^{-1} q^n\right)^{\frac{1}{2}},q^n \right)$, we have 
$$|\gamma \omega|_i=\dfrac{|\omega|_i}{|c \omega+d|^2}<|z_n-s|_i=q^{-n}.$$
For all the pairs $(c,d)$ satisfying at least one of the two conditions on $|c|$ or $|d|$, we have $|c \omega+d|>q^n$ and since $|\gamma \omega|_i<|z_n-s|_i$ we have $|z_n-\gamma \omega| \geq |z_n-s-\gamma \omega|_i=|z_n-s|_i$. So we obtain
\begin{align*}
|F_{\gamma,s}(z_n)-1| &\leq \dfrac{|z_n-s|C_{\omega,\alpha,s}}{|z_n-s|_i|c \omega+d|} & \mathrm{by} \ \eqref{majorationFg}\\
& \leq \dfrac{C_{\omega,\alpha,s}}{|c \omega+d|} \leq C_{\omega, \alpha,s}q^{-n}
\end{align*}
which is less than or equal to $\varepsilon$ when $n$ is large enough.
\end{proof}
\begin{lem }\label{lemcv3}
Let $\varepsilon>0$ be a real number. There exists $n_1\geq 0$ such that for all $n \geq n_1$, if $\gamma \in \Gamma$ satisfies $|\gamma \omega|_i>|z_n-s|_i$ then $|F_{\gamma,s}(z_n)-1|\leq \varepsilon$.
\end{lem }
\begin{proof}
Let $\varepsilon >0$ be fixed and $n \geq 0$ be large enough. If $\gamma \in \Gamma $ satisfies $|\gamma \omega|_i>|z_n-s|_i$ 
then $|s-\gamma \omega|\geq |\gamma\omega|_i>|z_n-s|_i=|z_n|_i=|z_n-s|$. 
By the ultrametric inequality, we have 
$|z_n-\gamma \omega|=|(z_n-s)+s-\gamma \omega|=|s- \gamma \omega|\geq \dfrac{\kappa_{\omega,s}}{|c \omega+d|}$ 
where $\kappa_{\omega,s}$ is independent of $\gamma \in \Gamma$ (see \citep[proof of Lemma 5.3.9]{gekelerrevjac96}). 
Then by \eqref{majorationFg}, we obtain 
$|F_{\gamma,s}(z_n)-1| \leq |z_n-s|C_{\omega, \alpha,s} \kappa_{\omega,s}^{-1}$, which is less than or equal to $\varepsilon$ if $n$ is large enough.
\end{proof}
\begin{coro}\label{corocv4}
Let $\varepsilon>0$. There exists $n_2 \geq 0$ such that for all $n \geq n_2$ and all $\gamma \in \Gamma$, we have
$|F_{\gamma,s}(z_n)-1|\leq \varepsilon$.
\end{coro}
\begin{proof}
Let $\varepsilon>0 $. We let $n_2=\max(n_0,n_1)$ where $n_0$ and $n_1$ are taken as in Lemma \ref{lemcv2} and \ref{lemcv3}.  We derive $|F_{\gamma,s}(z_n)-1|< \varepsilon$ if $|\gamma \omega|_i \neq |z_n-s|_i$. 
Moreover since $|\omega|=|\omega|_i\in q^{\Q-\Z}$, for all 
$\gamma=\begin{pmatrix} a&b\\c&d\end{pmatrix}\in \widetilde{\Gamma}$ we have $|c \omega+d|=\max(|c \omega|,|d|)$. Hence 
we get $|\gamma \omega|_i\in q^{\Q-\Z}$ for all $\gamma \in \widetilde{\Gamma}$, therefore 
$|\gamma \omega|_i \neq |z_n-s|_i=q^{-n}$.
\end{proof}
From Corollary~\ref{corocv4}, there exists $n_2$ such that for all $n \geq n_2$,
$$
\left|\Prod_{\gamma \in \widetilde{\Gamma}}\dfrac{s-\gamma \omega}{s-\gamma \alpha \omega}-\Prod_{\gamma \in  \widetilde{\Gamma}}\dfrac{z_n-\gamma \omega}{z_n-\gamma \alpha \omega} \right| \leq C\varepsilon,
$$
where the constant $C$ is as in \eqref{ConstC}. This concludes the proof of Theorem~\ref{Prodcusp}.
\section{Hecke operators}\label{HeckS}
From now on, the arithmetic subgroup $\Gamma$ will be $\Gamma_0(n)$ with $n \in A$ non constant and monic. 
Let $\varphi \in \Ht(\T,B)^\Gamma$ where $B=\Z,\Q,\R$ or $\mathbb{C}$. We consider $\varphi$ as a map on $\Gl(\K)$. Let $\mathfrak{m}=(m) \subset A$ be an ideal coprime to $(n)$. We set
$$
T_\mathfrak{m} (\varphi) (e)= \underset{ad=m,(a,n)=1}{\underset{\deg(d)<\deg(b)}{\underset{a \ \mathrm{monic}  }{\Sum_{a,b,d \in A}}}}\varphi \left( \begin{pmatrix}
a&b \\0&d 
\end{pmatrix}e
\right) \qquad (e \in E(\T)).
$$ 
Then we have $T_{\mathfrak{m}}(\varphi) \in \Ht(\T,B)^\Gamma$ and the operator $T_\mathfrak{m}$ is called 
the $\mathfrak{m}$-th Hecke operator. The following classical statements may be found for instance in \citep[Section~7]{MR1630606}.
%is proved in~\citep[Chapter 6, Section 5]{Serrecours} in the case of 
%the action of Hecke operators on set of lattices. The proof transfers easily to actions on the set of harmonic cochains.
%
\begin{prop}
The Hecke operators have the following properties:
\begin{enumerate}
\item Let $\mathfrak{p}\subset A$ be a prime ideal of A and $k $ be a nonnegative integer. Then $T_{\mathfrak{p}^k}$ is a polynomial in $T_\mathfrak{p}$.
\item Let $\mathfrak{m},\mathfrak{m'}$ be two coprime ideal of $A$, then $T_{\mathfrak{m}\mathfrak{m'}}=T_{\mathfrak{m}}T_{\mathfrak{m'}}$.
\item Hecke operators commute with each other.
\end{enumerate}
\end{prop}
It is also possible to define a Hecke action on $\overline{\Gamma}$ (see \citep[9.3]{gekelerrevjac96}). If $\mathfrak{p}=(P)$ is 
an ideal of $A$ coprime to $n$, we let $\tau_P=\begin{pmatrix}P&0\\0&1\end{pmatrix} \in \Gl(\K)$ and we define 
$$
\Delta_P=\Gamma \cap \tau_P \Gamma \tau_P^{-1}.
$$
For $\alpha \in \overline{\Gamma}$, we define $T_\mathfrak{p}( \alpha)$ as 
\begin{equation}\label{Heckgam}
T_\mathfrak{p}(\alpha)=\tau_P^{-1}\Prod_{\alpha_i \in \Delta_P \backslash \Gamma}\alpha_i\alpha \alpha_{\sigma(i)}^{-1}\tau_P
\end{equation}
where $\alpha_i$ runs through a system of representatives of $\Delta_P \backslash \Gamma$ and where $\sigma$ is the permutation 
of $\Delta_P \backslash \Gamma$ such that $\alpha_i \alpha \alpha_{\sigma(i)}^{-1} \in \Delta_P$.
\begin{lem }[{\citep[Lemma 9.3.2]{gekelerrevjac96}}]
Let $\mathfrak{a}=(a)$ be an ideal of $A$ coprime to $n$ and $\alpha \in \Gamma$. We have $j(T_\mathfrak{a}(\alpha))=T_\mathfrak{a}(j(\alpha))$. In other words, the action of Hecke operators commutes with the isomorphism $j$.
\end{lem }
\begin{prop}[{\citep[Lemma 9.3.3]{gekelerrevjac96}}]
The Hecke operators on $\overline{\Gamma}$ defined in \eqref{Heckgam} are self-adjoint with respect to the 
bilinear map $(\alpha, \beta) \mapsto c_\alpha(\beta)$.
\end{prop}
Let $J_\Gamma$ be the Jacobian of $\overline{M}_\Gamma$. It is an abelian variety over $K$ of dimension $g(\Gamma)$ 
where $g(\Gamma)$ is the genus of the modular curve $\overline{M}_\Gamma$.
\begin{def }[{\citep[VIII,1]{gekelerdrinfeldmod}}]
Let $\mathfrak{p}=(P)$ be a prime ideal of $A$ and $\Delta_P$, $\tau_P$ as above. The $\mathfrak{p}$-th Hecke operator 
on $\overline{M}_\Gamma$ is given by the correspondence
$$
T_\mathfrak{p}\omega=\Sum_{\alpha \in \Delta_P\backslash \Gamma}\tau_P^{-1} \alpha \omega, \ \omega \in \overline{M}_\Gamma(\C).
$$
\end{def }
The action of Hecke operators on $\overline{M}_\Gamma$ induces an action on $J_\Gamma$. Indeed if 
$D=\sum_{}a_i(\omega_i) \in J_\Gamma(\C)$ then $T_\mathfrak{p}(D)=\sum_{}a_i(T_\mathfrak{p}(\omega_i))$ 
(see \citep[Chapter VIII]{gekelerdrinfeldmod} for more details). The action of Hecke operators on $J_\Gamma(\C)$ induces 
the following action on $\C^{*}/t^{\Z}\simeq E(\C)$ (see \citep[9.4 and 9.5.6]{gekelerrevjac96})
\begin{equation}\label{hecke-fonctiontheta}
T_\mathfrak{m}u_\alpha(z)= \underset{ad=m,(a,n)=1}{\underset{\deg(d)<\deg(b)}{\underset{a\, \mathrm{monic}  }{\Prod_{a,b,d \in A}}}}u_\alpha \left(\dfrac{az+b}{d} \right).
\end{equation}
\section{Modular parametrization}\label{Modpar}
\begin{prop}[{\citep[7.3.3]{gekelerrevjac96}}]
The analytic group variety $\Hom(\overline{\Gamma},\C^{*})/c(\overline{\Gamma})$ carries the structure of an abelian variety $\K$-isomorphic to $J_\Gamma$.
\end{prop}
In other words, we have the exact sequence
\begin{equation}
1 \rightarrow \overline{\Gamma} \overset{c}{\rightarrow }\Hom(\overline{\Gamma},\C^{*}) \rightarrow J_\Gamma(\C) \rightarrow 0.
\end{equation}
The Hecke operators act on each of its terms and the map $\overline{c}\colon \overline{\Gamma} \rightarrow \C$ is compatible with 
their action (see \citep[9.3.3]{gekelerrevjac96}). The same holds for the projection 
$\Hom(\overline{\Gamma}, \C^{*}) \rightarrow J_\Gamma(\C)$.
%\begin{coro}\citep[7.5.2]{gekelerrevjac96}
%Let $u \in \Theta_m(\Gamma)$ have divisor $D$ on $\overline{M}_\Gamma$. Then $D=0$, In particular each $u \in \Theta_m(\Gamma)$ without zeros and poles is a holomorphic theta function $u_\alpha$, for some $\alpha \in \Gamma$.
%\end{coro}
%\smallskip

Let $E/K$ be an elliptic curve with split multiplicative reduction at the place $\infty=1/T$ of~$K$. Equivalently, $E$ has Tate parametrization
\begin{equation}
E(\K) \simeq \K^{*}/t^{\Z},
\end{equation}
for some $t \in \K^{*}$ with $|t|<1$. For each prime $\mathfrak{p}$, we set
$$
\lambda_\mathfrak{p}=
\begin{cases}
q^{\deg(\mathfrak{p})}+1- \card(\overline{E}(\F_{\mathfrak{p}})) & \mathrm{if} \ E \mathrm{ \ has \ good \ reduction  \ at} \ \mathfrak{p}, \\
1 & \mathrm{if \ }  E \mathrm{ \ has \ split\  multiplicative \ reduction \ at } \ \mathfrak{p}, \\
-1 & \mathrm{if \ }  E \text{ has non-split multiplicative reduction at} \ \mathfrak{p}, \\
0 & \mathrm{otherwise,}\end{cases}
$$
where $\F_\mathfrak{p}=A/\mathfrak{p}$ and $\overline{E}$ denote the reduction of $E$ at $\mathfrak{p}$. From  the work of Weil, Jacquet-Langlands, Grothendieck, Drinfeld and Zarhin, there exists a unique newform $\varphi \in \Hzo$ which is it is not divisible on $\Hzo$ and such that
$$
\left\{
\begin{array}{ll}
c(\varphi,1)=1 & \\
T_\mathfrak{m}(\varphi)=q^{\deg(\mathfrak{m} )}c(\varphi,\mathfrak{m})\varphi & \mbox{for all  ideal } \mathfrak{m}\subset A,\\
c(\varphi,\mathfrak{p})=q^{-\deg(\mathfrak{p})}\lambda_\mathfrak{p} & \mbox{for all  prime } \mathfrak{p} \subset A
\end{array}
\right.
$$
where $\deg(\mathfrak{m})=\log_q\left( \card(A/\mathfrak{m} )\right)$.

For $\mathfrak{m}=(m)$,  the coefficient $c(\varphi, \mathfrak{m})$ is the Fourier coefficient of $\varphi$ given by (see \citep[Section 3]{gekeler1995analytical})
$$
c(\varphi,\mathfrak{m})=q^{-1-\deg(\mathfrak{m})}\Sum_{y \in \Fq^* \backslash \pi \O /\pi^{2+\deg(\mathfrak{m})}\O}\varphi \left(\begin{pmatrix}
\pi^{2+\deg(\mathfrak{m})}&y\\0&1 
\end{pmatrix}\right) \nu(my)
$$
where $\nu \colon \K \rightarrow \Z$ is the character that maps the element $\sum_{i \in \Z}a_i\pi^i$ to $-1$ if $a_{1}=0$ and to $q-1$ otherwise.
%The cochain $\varphi$ is a newform (see Definition \ref{newfr}) and it is not divisible on $\Hzo$.
Let $\Gamma_\infty$ be the subgroup
$$
\Gamma_\infty=\left\{\begin{pmatrix}
a&b\\0&d
\end{pmatrix} \in \Gamma, \ a,b \in \Fq^*,d \in A\right\}.
$$
Recall that  $\varphi$ can be seen as a map on $Y^{+}(\Gamma_\infty \backslash \T)$ and its Fourier expansion is given by the following expression, for all $j\in\Z$ and $y \in \K \bmod \pi^j \O$ (see \citep{gekeler1995analytical} or \citep{gekeler1995improper}),
\begin{equation}
\varphi \left( \begin{pmatrix}
\pi^j & y \\0&1
\end{pmatrix}\right)= \Sum_{k=0}^{j-2}q^{k+2-j}\underset{\deg(f)=k}{\Sum_{f \in A, \ f \ \mathrm{monic}}}c(\varphi,(f))\nu(fy).
\end{equation}
Identifying $\Hzo$ with $\overline{\Gamma} $ by means of the isomorphism $j$ of Theorem~\ref{isomj}, we may consider $\varphi$ as an element of 
$\overline{\Gamma}$. 
\begin{prop}[{\citep[Proposition 9.5.1]{gekelerrevjac96}, \citep[Theorem 3.2]{gekeler1995analytical}}]
Let $\mathrm{ev}_\varphi$ be the evaluation map at $\varphi$ from $\Hom(\overline{\Gamma}, \C^{*})$ to $\C^{*}$. The subgroup $\mathrm{ev}_\varphi(c(\overline{\Gamma}))$ is isomorphic to $\Z$ and generated by a unique $t \in \K^{*}$ with $|t|<1$.
\end{prop}
Let $\Lambda=\{c_\varphi(\gamma), \gamma \in \Gamma\}=t^\Z$. We have the following commutative diagram
 \begin{equation}\label{diagc}
\xymatrix{
1 \ar[r] & \ \overline{\Gamma}  \ \ar[d]^{c_\varphi(.)} \ar[r]^{\overline{c} \qquad }& \Hom(\overline{\Gamma}, \C^{*}) \ar[d]^{\mathrm{ev}_\varphi} \ar[r] & J_\Gamma(\C) \ar[d]^{} \ar[r]& 0 \\
1 \ar[r] & \Lambda \ \ar[r] &  \C^{*}\ar[r] & \C^{*}/\Lambda \ar[r]& 0
}
\end{equation}
\begin{def }
The modular parametrization (also called Weil uniformization) is defined as
$$
\begin{array}{lccl}
\Phi \colon &\overline{M}_\Gamma(\C)& \longrightarrow &\C^{*}/\Lambda \simeq E_{\varphi}(\C)\\
&z & \longmapsto & u_\varphi(z),
\end{array}
$$
where $E_\varphi$ is called the Weil curve associated to the harmonic cochain $\varphi$. We note that $E_\varphi$ is a strong Weil curve in the sense that $\Phi $ can not be factorized through another Weil uniformization $\Phi'\colon \overline{M}_\Gamma (\C)\overset{\Phi'}{\rightarrow} E'(\C) \rightarrow E_\varphi(\C)$, where $E'$ is a Weil curve.
\end{def }

We set 
$\mu=\inf \{\langle \varphi, \psi \rangle>0, \ \psi \in \Hzo\}=\inf\{\log_q|c_\varphi(\gamma)|>0, \ \gamma \in \overline{\Gamma}\}$. Gekeler  has proved that $\mu=v_\infty(t)=-v_\infty(j_E)$ where $j_E$ is the $j$-invariant of $E_\varphi$ (\cite[Theorem~3.2, Corollary~3.19]{gekeler1995analytical}). Let us recall also his formula for the degree of the modular parametrization.

\begin{theo}[{\citep[Proposition 3.8]{gekeler1995analytical}}]\label{deg modular}
The degree of modular parametrization $\Phi$ is given by 
$$
\deg(\Phi)=\dfrac{\langle \varphi, \varphi \rangle}{\mu}.
$$
\end{theo}
We want to compute the values taken by the modular parametrization $\Phi$ at the cusps of $\overline{M}_\Gamma$. By the analogue of the Manin-Drinfeld theorem, proved in this setting by Gekeler \citep{gekeler2000note}, the subgroup of the Jacobian of $\overline{M}_\Gamma$ generated by the cusps is finite. As a consequence, the image of the cusps by $\Phi$ are torsion points of the elliptic curve. However Gekeler's statement provides no explicit bound for the size of the cuspidal subgroup in general. For our purpose, we prove an explicit bound on the order of such torsion points: the bound depends on the elliptic curve and is proved using Hecke operators and the same proof principle as Manin-Drinfeld.
\begin{theo}\label{ptetors}
Let $s \in \Gamma \backslash \Pj^1(K)$ be a cusp of $M_\Gamma$. Then $\Phi(s)$ is a torsion point of~$E$ and its order divides $\card(\overline{E}(\F_\mathfrak{p}))$ for all $\mathfrak{p}=(P)$ with $P\equiv 1 \bmod n$.
\end{theo}
\begin{proof}
Let $z \in \overline{M}_\Gamma(\C)$ and let $\mathfrak{p}=(P)$  be a prime ideal of $A$ such that $P\equiv 1 \bmod n$. By the analogue of Dirichlet's theorem (\citep[Theorem~4.7]{Rosennumberthff}), such a polynomial $P$ exists. We have
$$
T_\mathfrak{p}u_\varphi(z)=u_\varphi(Pz) \underset{\deg(j)<\deg(P)}{\Prod_{j \in A}}u_\varphi \left(\dfrac{z+j}{P} \right).
$$
Since $T_\mathfrak{p}$, as any other Hecke operator, commutes with the map 
$f \in \Hom(\overline{\Gamma}, \C) \mapsto f(\varphi) \in \C^*$, we have by \eqref{hecke-fonctiontheta},
\begin{equation}\label{Heckell}
u_\varphi(z)^{\lambda_\mathfrak{p}}=u_\varphi(Pz)\underset{\deg(j)<\deg(P)}{\Prod_{j \in A}} u_\varphi \left(\dfrac{z+j}{P}\right).
\end{equation}
Let $s \in \Gamma \backslash \Pj^1(K)$ be a cusp. We let $\mathcal{S}_P=\{j \in A, \ \deg(j)<\deg(P)\}$. Since $P \equiv 1 \bmod n$ and by arguments similar to  \citep[Proposition 2.2.3]{cremona97algo}, there exist
$Q_P$ and $(Q_j)_{j \in \mathcal{S_P}} \subset \Gamma$ such that $Q_P s=Ps$ and $Q_j s=\dfrac{s+j}{P}$ for all 
$j \in \mathcal{S}_P$. We take $z=s$ in \eqref{Heckell}, divide by $u_\varphi(s)^{q^{\deg(P)+1}}$ and obtain
\begin{align*}
u_\varphi(s)^{\lambda_{\mathfrak{p}}-(q^{\deg(P)}+1)} &= \dfrac{u_\varphi(Ps)}{u_\varphi(s)}\underset{\deg(j)<\deg(P)}{\Prod_{j \in A}} \dfrac{u_\varphi \left(\frac{s+j}{P}\right)}{u_\varphi(s)}, \\
&=\dfrac{u_\varphi(Q_P s)}{u_\varphi(s)}\underset{\deg(j)<\deg(P)}{\Prod_{j \in A}} \dfrac{u_\varphi (Q_j s)}{u_\varphi(s)} \in \Lambda.
\end{align*}
So $\Phi(s)$ is a torsion point of $E(\C)$ whose order divides $\card(\overline{E}(\F_{\mathfrak{p}}))$.
\end{proof}
\begin{rem}
In specific cases, explicit bounds for the size of the cuspidal subgroup of the Jacobian are known, for example when $n$ is irreducible by P\'al (\citep{palontheeis}) or $\deg(n)=3$ by Papikian and Wei (\citep{papikian2016eisenstein}). These bounds, which in contrast with Theorem~\ref{ptetors} depend only on $n$, also are upper bounds for the order of torsion points on the elliptic curve coming from evaluating $\Phi$ at cusps. However Theorem~\ref{ptetors} requires no assumption on $n$.
\end{rem}
\section{Examples}\label{Appl}
\subsection{Example in \texorpdfstring{$\F_2(T)$}{F2(T)} with \texorpdfstring{$n=T^3$}{n=T3}}\label{examp1}
We consider the elliptic curve $E/\F_2(T)$ defined by 
$$ y^2+Txy=x^3+T^2.$$
Its conductor is $\mathfrak{n}=T^3 \infty$. The elliptic curve $E$ is isomorphic over $\F_2(T)$ to the strong Weil curve given
in  \citep[Theorem 2.1 (a)]{schweizerstrong11}. We set $n=T^3$ and $\Gamma=\Gamma_0(n)$. The genus of $\overline{M}_\Gamma$ 
is $1$, thus the abelian group $\Hzo$ has rank one. Let $\varphi_E \in \Hzo$ be the harmonic cochain attached to $E$ as described in Section~\ref{Modpar}. We 
can identify $\varphi_E$ with its Fourier expansion: indeed since $\varphi_E$ has finite support modulo~$\Gamma$, it is sufficient 
to know its values on the finite number of edges in $(\Gamma \backslash \T)^0$. 
Furthermore we can compute the inverse image of $\varphi_E$ by the isomorphism $j$ of Theorem~\ref{isomj}. For this purpose, we compute 
the inverse image of a $\Z$-basis 
of $\Hzo$ with the maximal subtree method in $\Gamma \backslash \T$ as in \citep[proof of Lemma 3.3.3]{gekelerrevjac96}. The element $\alpha \in \overline{\Gamma}$ such that $j(\alpha)=\varphi_E$ 
that we obtain is then
$$
\alpha =\begin{pmatrix}
T^2+T+1&1\\T^3&1+T
\end{pmatrix}.
$$ 
Note that $\alpha $ generates $\overline{\Gamma}$ since $\varphi_E$ generates $\Hzo$, hence 
$\Lambda=\{c_\alpha(\gamma), \gamma \in \overline{\Gamma}\}=\langle c_\alpha(\alpha)\rangle$. In order to determine the lattice 
$\Lambda$, we just need  to compute $c_\alpha(\alpha)$. The coefficient $c_\alpha(\alpha)$ is given by (see \eqref{coeffmul})
\begin{equation}\label{calpha1}
c_\alpha(\alpha)=\Prod_{\gamma \in \widetilde{\Gamma}}\dfrac{\alpha \infty-\gamma \omega}{\alpha \infty-\gamma \alpha \omega},
\end{equation}
and this product is independent of $\omega \in \Omega$. We choose $\omega=T^{-2}\rho$ with $\rho \in \C^{*}$ satisfying 
$\rho^2+\rho+1=0$.
We obtain a suitable approximation of $c_\alpha(\alpha)$ by restricting the product~\eqref{calpha1} to the set 
$$
\left\{\gamma \in \widetilde{\Gamma},\left|\dfrac{\alpha \infty-\gamma \omega}{\alpha \infty-\gamma \alpha \omega}-1\right|\geq \frac{1}{16}\right\}.
$$ 
This set is then determined by similar arguments as those in the proof of Proposition~\ref{appox1pts}. 
We obtain
$c_\alpha(\alpha)=\pi^{-4}+\delta_{c_\alpha(\alpha)}$ where $\delta_{c_\alpha(\alpha)} \in \K^*$ is such that 
$|\delta_{c_\alpha(\alpha)}|<1$. We deduce that a generator of $\Lambda$ is  $t=c_\alpha(\alpha)^{-1}=\pi^4+\delta_{t_\alpha}$ with 
$|\delta_{t_\alpha}|<2^{-8}$.

The modular curve $\overline{M}_\Gamma$ has four cusps: $\infty, 0, \frac{1}{T}, \frac{1}{T^2}$. By Theorem~\ref{deg modular}, the degree of the modular 
parametrization is $1$ and $\Phi$ is therefore an isomorphism. For a cusp $c\in \{\infty, 0, \frac{1}{T}, \frac{1}{T^2}\}$, we want to evaluate $u_\alpha(c)$ modulo $\Lambda$. Using Theorem~\ref{ptetors}, we compute $\card(\overline{E}(\F_\mathfrak{p}))$ for all primes  $\mathfrak{p}=(P)$ with $\deg(P)\leq 15$ and obtain that the image of $c$ is a torsion point in $E$ of order dividing $16$. Note that in this example, the 
bound given by Papikan-Wei in \citep{papikian2016eisenstein} is $4$, which is better than ours. 
\begin{prop}\label{appox1pts}
We have the following approximations:
\begin{align*}
& u_\alpha(0)=\pi+\upsilon, &&  \mbox{ with } |\upsilon|<2^{-5}, \\ 
&u_\alpha\left(\frac{1}{T}\right)=\pi^{-1}+\delta_{1/T},  && \mbox{ with }|\delta_{1/T}|<2^{-3}, \\
&u_\alpha\left(\frac{1}{T^2} \right)=\pi^2+\delta_{1/T^2}, && \mbox{ with } |\delta_{1/T^2}|<2^{-6}.
\end{align*}
In particular, the points $\Phi(0)$ and $\Phi(\frac{1}{T}) $ have order $4$ in $E(\K)$, and the point $\Phi(\frac{1}{T^2})$ has  order 
$2$ in $E(\K)$.
\end{prop}
\begin{proof}
We only give details for the computation of $u_\alpha$ at $s=0$. Recall that by Theorem~\ref{Prodcusp}, the value $u_\alpha(0)$ is given by 
$$
u_\alpha(0)=\Prod_{\gamma \in \widetilde{\Gamma}}\dfrac{0-\gamma \omega}{0-\gamma \alpha \omega}.
$$
As above, we choose $\omega=T^{-2}\rho$ with $\rho^2+\rho+1=0$. We then observe that $|\omega|=|\omega|_i=2^{-2}$, 
$|\alpha \omega|=2^{-1}$ and $|\alpha \omega|_i=2^{-4}$. 
First let $\gamma=\begin{pmatrix}
a&b\\c&d
\end{pmatrix}$ be an element of $\widetilde{\Gamma}$. We have 
\begin{align*}
\left|\dfrac{-\gamma \omega}{-\gamma \alpha \omega}\right|&=\left|1+\dfrac{\gamma \alpha  \omega-\gamma \omega}{-\gamma\alpha \omega}\right| 
 \leq \max \left(1,\left|\dfrac{\gamma \alpha \omega -  \gamma \omega}{- \gamma \alpha \omega}\right| \right),
\end{align*}
and by \eqref{gammaomg},
\begin{equation}\label{eqgamma}
\left|\dfrac{\gamma \alpha \omega -  \gamma \omega}{- \gamma \alpha \omega}\right|=\dfrac{|\alpha \omega-\omega|}{|-\gamma \alpha \omega||c \omega+d||c \alpha \omega+d|}=\dfrac{2^{-1}}{|a \alpha \omega+b||c \omega+d|}.
\end{equation}
But we have
\begin{align*}
|a \alpha \omega+b|&=\dfrac{|(a(1+T+T^2)+bT^3)\omega+b(1+T)+a |}{|T^3 \omega +1+T|} \\
& \geq 2^{-1}\left|(a(1+T+T^2)+bT^3)\omega+b(1+T)+a \right| \\
&\geq 2^{-1}.
\end{align*}
Hence if $c \neq 0$
\begin{align*}
\left|\dfrac{\gamma \alpha \omega -  \gamma \omega}{- \gamma \alpha \omega}\right| &\leq \dfrac{1}{|c\omega+d |} \leq \dfrac{2^2}{|c|}
\leq 2^{-1}.
\end{align*}
If $c=0$, then $a=d=1$ and $|c \omega+d|=1$ and we have
$$
\left| \dfrac{\gamma \alpha \omega- \gamma \omega}{-\gamma \alpha \omega}\right|=\dfrac{2^{-1}}{|\alpha \omega+b|} \leq 1.
$$
For all subsets $S \subset \widetilde{\Gamma}$, we deduce
\begin{equation}\label{majprodzero}
\left|\Prod_{\gamma \in S}\frac{-\gamma \omega}{-\gamma \alpha \omega}\right|\leq 1.
\end{equation}
By this inequality, we see that to compute $u_\alpha(0)$ at a precision at most $\varepsilon$, it is 
sufficient to compute the product 
$$
\underset{|G_{\gamma,\omega}(0)-1|\geq \varepsilon}{\Prod_{\gamma \in \widetilde{\Gamma}}}G_{\gamma,\omega}(0),
$$
where $G_{\gamma,\omega}(z)=\dfrac{z-\gamma \omega}{z-\gamma \alpha \omega}$. If $|c|>\frac{2^2}{\varepsilon}$, we have
\begin{align*}
|G_{\gamma,\omega}(0)-1| &= \dfrac{2^{-1}}{|a \alpha \omega+b||c \omega+d|} & \text{(by \eqref{eqgamma})} \\& \leq \dfrac{2^2}{|c|} < \varepsilon.
\end{align*}
If $|c|\leq \frac{2^2}{\varepsilon}$ and if  $|d|>\frac{2}{\varepsilon} $, we have $|G_{\gamma,\omega}(0)-1| <\varepsilon $. Now for 
the pairs $(c,d)$ not satisfying the condition ($|c|> \frac{2^2}{\varepsilon}$ or $|d|>\frac{2}{\varepsilon}$) then,  if $c\neq 0$ we 
have $|c \omega+d|\geq |c||\omega|_i\geq 2^{1}$, and then
\begin{align*}
|G_{\gamma,\omega}(0)-1|&\leq \dfrac{2^{-2}}{|a \alpha \omega+b|} < \dfrac{2^2}{|a|} 
 < \varepsilon & \mathrm{if} \ |a|>\frac{2^2}{\varepsilon}.
\end{align*}
And if $c=0$ (so that $a=d=1$), then if $|b|>\frac{2^{-1}}{\varepsilon}$, we obtain
$$
|G_{\gamma ,\omega}(0)-1|=\dfrac{2^{-1}}{| \alpha \omega+b|} \leq \frac{2^{-1}}{|b|} < \varepsilon. 
$$
Note that, except for the case $c=0$, for fixed $a, c,d$ there exists at most one value of $b \in A$ such that 
$
\begin{pmatrix}
a&b\\c&d
\end{pmatrix} \in \Gamma
$. Hence, to approximate $u_\alpha(0)$ at precision at most~$\varepsilon$ we must compute the product 
$$
\Prod_{\gamma \in \widetilde{\Gamma}_\varepsilon}G_{\gamma,\omega}(0),
$$
where
$$
\widetilde{\Gamma}_\varepsilon=\left\{\gamma =\begin{pmatrix}
a&b\\c&d
\end{pmatrix} \in \widetilde{\Gamma},c\neq 0,|c| \leq \frac{2^2}{\varepsilon}, |d| \leq  \frac{2}{\varepsilon}, |a| \leq \frac{2^2}{\varepsilon} \right\} \cup \left\{\gamma \in \widetilde{\Gamma},c=0, |b| \leq \dfrac{1}{2 \varepsilon} \right\}.
$$
The computation of the finite product $\Prod_{\gamma \in \widetilde{\Gamma}_\varepsilon}G_{\gamma,\omega}(0)$ for $\varepsilon = 2^{-5}$ has been done 
with the software Pari-GP \citep{PARI2}. We obtain $u_\alpha(0)= \pi + \upsilon$ with 
$\upsilon \in \K^*$ satisfying $|\upsilon|<2^{-5}$, thus $u_\alpha(0)^4=t+\upsilon^4$. We conclude that $\Phi(0)$ has order $4$.
%\smallskip
%~\\

For $s=\frac{1}{T}$, we obtain $u_\alpha(\frac{1}{T})=\pi^{-1}+\delta_{1/T}$ with $\delta_{1/T}\in \K^*$ satisfying 
$|\delta_{1/T}|<2^{-3}$: we deduce that the point $\Phi(1/T)$ has order $4$ on $E$. Finally, we have 
$u_\alpha \left(\frac{1}{T^2}\right)=\pi^2+\delta_{1/T^2}$ with $|\delta_{1/T^2}| \leq 2^{-6}$, hence the point 
$\Phi\left(\frac{1}{T^2}\right)$ has order $2$.
\end{proof}
\begin{rem}
In this example we can also give an explicit bound on the cardinal of the set~$\widetilde{\Gamma}_\varepsilon$.
The set $S=\{c \in \F_2[T]: c \equiv 0 \bmod T^3,\, \deg(c)\leq 2+ \lfloor \log_2(\frac{1}{\varepsilon})\rfloor\}$ has cardinal $2^{-\lfloor \log_2(\varepsilon) \rfloor}$.
Let $c \in S$. Consider the set $S_c$ of polynomials $d \in \F_2[T]$ such that there exists a matrix  $\begin{pmatrix} a&b\\c&d
\end{pmatrix} \in \widetilde{\Gamma}$: an upper bound on its cardinal is 
$$\card (\{ d \in \F_2[T]: \deg(d)\leq 1-\lfloor\log_2(\varepsilon) \rfloor, \, \pgcd(d,T)=1\})=2^{2-\lfloor \log_2(\varepsilon) \rfloor}.$$
Then for $c \in S$ and $d \in S_c$, the number of all possible polynomials $a$ such that there exists $b\in A$ with $\begin{pmatrix}a&b\\c&d\end{pmatrix} \in \widetilde{\Gamma}$ is bounded by
$$\card( \{a \in \F_2[T], \pgcd(a,T)=1 \ \mathrm{and} \ \deg(a) \leq 2-\log_2(\varepsilon)\})=2^{2-\lfloor \log_2(\varepsilon)\rfloor}.$$
Hence we get $\card(\widetilde{\Gamma}_\varepsilon) \leq 2^{4-3 \lfloor \log_2(\varepsilon)\rfloor}$.
\end{rem}
\subsection{Example in \texorpdfstring{$\F_3(T)$}{F3(T)} with \texorpdfstring{$n=T^3-T^2$}{n=T3-T2}}
We consider the elliptic curve $E/\F_3(T)$ defined by 
$$
E \colon y^2=x^3+T(T+1)x^2+T^2x.
$$
Its conductor is $\infty (T^3-T^2)$. Again the curve $E$ is a strong Weil curve, see \citep[Theorem 2.4 (b)]{schweizerstrong11}.
Let $n=T^3-T^2$ be the finite part of conductor of $E$ and let $\Gamma=\Gamma_0(T^3-T^2)$. Here the genus of the curve 
$\overline{M}_\Gamma$ is 2 hence the abelian group $\Hzo$ has rank two. 
Let $\varphi_E$ be the harmonic cochain associated to $E$. Let $\alpha$ the inverse image of $\varphi_E$ by $j$. By the same method as in the first example, we find
$$
\alpha=\begin{pmatrix}
T^3+2T^2+2T+2&2T+1\\T^4+T^3+T^2&2T^2+2T+1
\end{pmatrix}.
$$
We compute a generator $t \in \K^*$ of the lattice $\Lambda=\{c_\alpha(\gamma), \gamma \in \widetilde{\Gamma}\}=t^\Z$. We obtain the following approximation for $t$:
$$
t=\pi^4+2\pi^5+O(\pi^6).
$$ 
The modular curve $\overline{M}_\Gamma $ has six cusps: $\infty,0,\frac{1}{T}, \frac{1}{T^2}, \frac{1}{T-1}, \frac{1}{T^2-T}$. The 
Tate period of~$E$ is given by (see \citep[Section 2 of Appendix A]{silverman2013advanced}) 
$$t_E=\Sum_{n\geq 1}d(n)\dfrac{1}{j_E^n}   =\pi^4+2\pi^5+\pi^7+2\pi^8 +\delta_{t_E},
$$
where $(d(n))_{n \geq 1}$ are the coefficients of the series for the reciprocal function of $j$, and $\delta_{t_E} \in \K^*$ with 
$|\delta_{t_E}|<3^{-10}$. We deduce that $t_E=t$.
The degree of the modular parametrization is $2$. As before, we compute $\card(\overline{E}(\F_\mathfrak{p}))$ for  prime ideals 
$\mathfrak{p}=(P)$ such that $\deg(P)\leq 10$ and $P \equiv 1 \bmod n$. By Theorem~\ref{ptetors}, we obtain that the order of image of a cusp divides 
$8$. This bound is sharper than the one of Papikian-Wei \citep{papikian2016eisenstein} which is $24$ in this example. 
\begin{prop}
We have the following approximations:
\begin{align*}
&u_\alpha(0)=2\pi^2+2\pi^3+\vartheta_0, & & \mbox{ with } |\vartheta_0|<3^{-3}, \\
&u_\alpha \left( \frac{1}{T}\right)=\pi^{-1}+1+\pi+\vartheta_{1/T}, & & \mbox{ with } |\vartheta_{1/T}|<3^{-1}, \\
&u_\alpha\left(\frac{1}{T^2}\right)=\pi^4+2\pi^5+\vartheta_{1/T^2} & & \mbox{ with } |\vartheta_{1/T^2}|<3^{-5}, \\
&u_\alpha \left(\frac{1}{T-1}\right)= \pi^{-2}+\pi^{-1}+\vartheta_{1/(T-1)}, & &\mbox{ with } |\vartheta_{1/(T-1)}|<1, \\
&u_\alpha\left( \frac{1}{T^2-T} \right)=\pi^3+\vartheta_{1/(T^2-T)}, & & \mbox{ with } |\vartheta_{1/(T^2-T)}|<3^{-5}.
\end{align*}
In particular, we get
\begin{enumerate}
\item the points $\Phi\left(\frac{1}{T}\right)$ and $ \Phi \left( \frac{1}{T^2-T}\right) $ have order $4$ in $E(\K)$,
\item the point $\Phi\left( \frac{1}{T^2} \right)$ has order 1 in $E(\K)$, 
\item  the point $\Phi(0)$ and $ \Phi\left( \frac{1}{T-1}\right)$ has order $2$ on $E(\K)$.
\end{enumerate}
\end{prop}
\begin{proof}
As before for $s \in \Gamma \backslash \Pj^1(K)$ and $\gamma \in \Gamma$, we let 
$G_{\gamma, \omega}(s)=\dfrac{s-\gamma \omega}{s-\gamma \alpha \omega}$. 
%First, we find the set, $\mathcal{S}_1$, 
%of matrices $\gamma \in \widetilde{\Gamma}$ such that $|G_{\gamma,\omega}(s)-1|>1$. 
As a system of 
representatives for $\tilde{\Gamma}$, we choose the matrices $\begin{pmatrix}
a&b\\c&d
\end{pmatrix}\in \Gamma$ with monic $a$. For $\varepsilon \leq 1$, let 
$$
\mathcal{S}_\varepsilon = \{\gamma \in \widetilde{\Gamma}, |G_{\gamma,\omega}(s)-1|\geq \varepsilon\}.
$$ 
We have  $\left|\Prod_{\gamma _\in \mathcal{S}_\varepsilon}G_{\gamma,\omega}(s)\right| \leq \left|\Prod_{\gamma \in \mathcal{S}_1}G_{\gamma,\omega}(s)\right|$ and  
$$
\left| \Prod_{\gamma \in \widetilde{\Gamma}}G_{\gamma,\omega}(s)-\Prod_{\gamma \in \mathcal{S}_\varepsilon}G_{\gamma,\omega}(s)\right|< \left|\Prod_{\gamma \in \mathcal{S}_1}G_{\gamma,\omega}(s)\right|\varepsilon.
$$
Indeed if $g_{\gamma,\omega}(s)=G_{\gamma, \omega}(s)-1$, we have 
$$
\left|\Prod_{\gamma \in \widetilde{\Gamma}-\mathcal{S}_\varepsilon}G_{\gamma,\omega}(s)-1\right|=\left|\Prod_{\gamma \in \widetilde{\Gamma}-\mathcal{S}_\varepsilon}(1+g_{\gamma,\omega}(s))-1\right|
$$
and we can easily see that 
$\Prod_{\gamma \in \widetilde{\Gamma}-\mathcal{S}_\varepsilon}(1+g_{\gamma,\omega}(s))=1+\delta_{\omega,\varepsilon}(s)$ 
with ~$|\delta_{\omega,\varepsilon}(s)|<\varepsilon$. 
%\smallskip
%~\\

We give details only for the cusp $s=\frac{1}{T}$. We choose $\omega=T^{-2}\rho$ with $\rho^2+1=0$. Note that 
$|\omega|=|\omega|_i=3^{-2}$, $|\alpha \omega|=3^{-1}$ and $|\alpha \omega|_i= 3^{-4}$. Let $\gamma=\begin{pmatrix}
a&b \\ c&d
\end{pmatrix} \in \Gamma$. We have 
$$|s-\gamma \alpha \omega|=\dfrac{\left|(c-aT)\alpha \omega+ (d-bT) \right|}{|T||c \alpha \omega+d|}.
$$
Moreover we have
$$
|(c-aT)\alpha \omega+d-bT| \geq 3^{-2}.
$$
Hence we obtain by \eqref{gammaomg} that if $|c|>\frac{3^4}{\varepsilon}$
$$
|G_{\gamma, \omega}(s)-1|\leq \dfrac{3^3 |\alpha \omega-\omega|}{|c \omega+d|} \leq \dfrac{3^2}{|c||\omega|_i}\leq \dfrac{3^4}{|c|} <\varepsilon.
$$
Furthermore, if $|c|\leq \frac{3^4}{\varepsilon}$ and
$|d|>\frac{3^2}{\varepsilon}$, we have $|c\omega+d|=|d|$ and we get 
$$
|G_{\gamma(s), \omega}-1|=\dfrac{3^2}{|d|}< \varepsilon.
$$
If $(c,d)$ does not satisfy the condition ($|c|>\frac{3^4}{\varepsilon}$ or $|d|>\frac{3^2}{\varepsilon}$) and if $c \neq 0$, we 
note that when $|a|> \frac{3^5}{\varepsilon}$, we have $|a \alpha \omega+b|>|c \alpha\omega+d||s|$. We then obtain
$$
|s-\gamma \alpha \omega|=\left|s-\dfrac{a\alpha \omega+b}{c \alpha \omega+d}\right|=\dfrac{|a \alpha \omega+b|}{|c \alpha \omega +d|}.
$$
Since $|c \omega+d| \geq |c||\omega|_i \geq 3^1$, we get by \eqref{gammaomg} 
\begin{align*}
|G_{\gamma, \omega}(s)-1| \leq \dfrac{|\alpha \omega- \omega|}{|a \alpha \omega+b||c \omega+d|} \leq \dfrac{3^2}{|a|} < \varepsilon, & &\mathrm{if } \ |a| >  \frac{3^5}{\varepsilon}.
\end{align*}
If $c=0$, we have $a=1$ and $d \in \F_3^{*}$, so we get if $b\neq 0$, $|s-\gamma \alpha \omega|=|s-\frac{\alpha \omega}{d}-\frac{b}{d}|=|b|$. Hence 
\begin{align*}|G_{\gamma,\omega} (s)-1|=\dfrac{|\alpha \omega-\omega|}{|b||c\omega+d|}\leq  \dfrac{3^{-2}}{|b|}< \varepsilon, & &\mathrm{if} \ |b|>\frac{3^{-2}}{ \varepsilon}.
\end{align*}
Here the product $\Prod_{\gamma \in \mathcal{S}_1}G_{\gamma,\omega}(s)$ has norm $3$. So if we want to approximate the 
product $\Prod_{\gamma \in \widetilde{\Gamma}}G_{\gamma,\omega}(s)$ with precision at most $\varepsilon$, we need to consider the product restricted to the set $\mathcal{S}_{\varepsilon/3}$.

For $\varepsilon = 3^{-1}$, the computation of $u_\alpha \left( \frac{1}{T}\right)$ at precision at most $3^{-1}$ gives $u_\alpha \left(\frac{1}{T}\right)=\pi^{-1}+1+\pi +\vartheta_{1/T}$, with $\vartheta_{1/T} \in \K^*$ satisfying $|\vartheta_{1/T}|<3^{-1}$. For the other cusps we get
\begin{align*}
&u_\alpha(0)=2\pi^2+2\pi^3+\vartheta_0, & & \mbox{ with } |\vartheta_0|<3^{-3}, \\
&u_\alpha\left(\frac{1}{T^2}\right)=\pi^4+2\pi^5+\vartheta_{1/T^2} & & \mbox{ with } |\vartheta_{1/T^2}|<3^{-5}, \\
&u_\alpha \left(\frac{1}{T-1}\right)= \pi^{-2}+\pi^{-1}+\vartheta_{1/(T-1)}, & & \mbox{ with } |\vartheta_{1/(T-1)}|<1, \\
&u_\alpha\left( \frac{1}{T^2-T} \right)=\pi^3+\vartheta_{1/(T^2-T)}, & & \mbox{ with }  |\vartheta_{1/(T^2-T)}|<3^{-5}.
\end{align*}
This proves the statement.
\end{proof}
\begin{rem}
We can also provide an upper bound on the cardinal of the set $\mathcal{S}_{1/9}$ needed to compute 
$u_\alpha\left(\frac{1}{T}\right)$ at precision at
most $\varepsilon = 3^{-1}$. We have 
\begin{align*}
\card (\{ c\in \F_3[T],\deg(c) \leq 7 \ \mathrm{and} \ (T^3-T^2)\mid c \})=3^4,\\
\card \left(\{d \in \F_3[T], \deg(d)\leq 4 \ \mathrm{and} \ \pgcd(d,T^3-T^2)\in \F_3^*\}\right) \leq 3^4.\end{align*} So the number of 
pairs $(c,d)$ such that $(c,d)$ can be lifted to  a matrix in $\widetilde{\Gamma}$ is bounded by $3^8$. We also have 
$$
\card \left(\{a \in \F_3[T], \ \deg(a)\leq 7, \ a \ \mathrm{monic} \ \mathrm{and} \, \pgcd(a,T^3-T^2) \in \F_3^{*}\}\right)\leq 3^7.
$$ 
We conclude that $\card(\mathcal{S}_{1/9})\leq 3^{15}$.
\end{rem}

\bibliographystyle{apalike}
\bibliography{bibliographie}
\addcontentsline{toc}{section}{References}

\end{document}